\documentclass{amsproc}
\usepackage{epsfig,amscd,amssymb,amsmath,amsfonts}
\newtheorem{theorem}{Theorem}[section]
\newtheorem{lemma}[theorem]{Lemma}
\newtheorem{proposition}{Proposition}[section]
\theoremstyle{definition}
\newtheorem{definition}[theorem]{Definition}
\newtheorem{corollary}[theorem]{Corollary}
\newtheorem{conjecture}[theorem]{Conjecture}
\newtheorem{example}[theorem]{Example}

\theoremstyle{remark}
\newtheorem{remark}[theorem]{Remark}

\numberwithin{equation}{section}

\begin{document}

\title{Operations on the Secondary  Hochschild  Cohomology}
\author{Mihai D. Staic}
\address{
Department of Mathematics and Statistics, Bowling Green State University, Bowling Green, OH 43403 }
\address{Institute of Mathematics of the Romanian Academy, PO.BOX 1-764, RO-70700 Bu\-cha\-rest, Romania.}
\thanks{This work was partially supported by a grant of the Romanian National Authority for Scientific Research, CNCS-UEFISCDI, project number PN-II-ID-PCE-2011-3-0635, contract nr. 253/5.10.2011.}
%\curraddr{}
\email{mstaic@gmail.com}


\author{Alin Stancu}
\address{Department of Mathematics, Columbus State  University, Columbus, GA 31907}
\email{stancu\_alin1@columbusstate.edu}

%    General info
\date{January 1, 1994 and, in revised form, June 22, 1994.}

%\dedicatory{I thank   }


\begin{abstract} We show that the secondary Hochschild cohomology  associated to a triple $(A,B,\varepsilon)$ has several of the properties of the usual Hochschild cohomology. Among others, we prove the existence of the cup and Lie products, discuss the connection with extensions of $B$-algebras, and give a Hodge type decomposition of the secondary Hochschild cohomology.

\end{abstract}

\maketitle

%\section*{This is an unnumbered first-level section head}

%
%%%%%%%%%%%%%%%%%%%%%%%%%%%%%%%%%%%%%%%%%%%%%%%%%%%

%%%%%%%%%%%%%%%%%%%%%%%%%%%%%%%%%%%%%%%%%%%%%%%%%%%%
\section*{Introduction}

Hochschild  cohomology was introduced by Hochschild in \cite{h} in order to study extensions of associative algebras over a field and to characterize the separability of this class of algebras. In the same paper (written while he was a draftee serving in the army) he defined for any associative algebra $A$ the cup product of cochains with coefficients in $A$. From Hochschild's definition it follows easily that the cup product on the cochains  descends  to one on the Hochschild cohomology $H^\bullet(A, A)$. Almost twenty years later Gerstenhaber  proved in \cite{g1} that at the cohomology level the cup product is graded commutative. He also defined a Lie product whose properties, when combined with those of the cup product, determine on $H^\bullet(A, A)$ a rich algebraic structure which is now called a Gerstenhaber algebra (or $G$-algebra). $G$-algebra structures appear in other contexts of which we mention here the exterior algebra of a Lie algebra, the differential forms on a Poisson manifold, and the Hochschild cohomology of presheaves of algebras. In this paper we show that on the secondary Hochschild cohomology we can define a cup product and a Lie product, which naturally extend those on the Hochschild cohomology.

Consider a $B$-algebra $A$ determined by the $k$-algebra homomorphism $\varepsilon:B\to A$. The secondary Hochschild cohomology $H^\bullet((A,B,\varepsilon);A)$ was introduced in \cite{sta} in order to study  the $B$-algebra structures on $A[[t]]$. It was proved there that a $B$-algebra structure on $A[[t]]$ is determined by a family of products $m_{\alpha}:A[[t]]\otimes A[[t]]\to A[[t]]$ that must satisfy a generalized associativity condition.  For $a$, $b\in A$ and $\alpha\in B$ we have $m_{\alpha}(a\otimes b)=\varepsilon(\alpha)ab+c_1(a\otimes b\otimes \alpha)t+...$. Just like in the case of deformations of algebras, $c_1$ is a 2-cocycle that gives the deformation $mod \;t^2$. Its class $c_1\in H^2((A,B,\varepsilon);A)$ is determined by the isomorphism class of the $B$-algebra $A[[t]]$.   Moreover, if we assume that $m_{\alpha}$ is  associative $mod\; t^{n+1}$ then the obstruction to extend it to an associative product $mod\;t^{n+2}$ is the vanishing of the element $c_1\circ c_n+c_2\circ c_{n-1}+...+c_n\circ c_1$ in $H^3((A,B,\varepsilon);A)$.

The paper is organized in five sections. In the first section we define the secondary Hochschild cohomology. In the second we introduce the cup and Lie products for the secondary cohomology and then prove some of their properties. In the third section we discuss the connection between extensions of $B$-algebras $0\to M\to X\to A\to 0$ with $M^2=0$ and $H^2((A,B,\varepsilon);M)$. In the forth  we give a Hodge type decomposition, in characteristic 0, for the secondary cohomology, one  that it is consistent with the Hodge decomposition of the usual Hochschild cohomology. Finally, in the fifth section we investigate the (cup and bracket preserving) natural map $\Phi:H^n((A,B,\varepsilon);A)\to H^n(A,A)$. More precisely, we present examples which show that in general $\Phi$ is neither surjective nor injective. Our examples deal with subalgebras of the ring of polynomials. We show that requiring $\Phi_2$ to be injective is equivalent to the Jacobian problem stated in \cite{W}, a question first posed by Ott-Heinrich Keller in 1939.

\section{Preliminaries}
\subsection{Hochschild  Cohomology of an algebra $A$}

In this paper $k$ is a field, $\otimes=\otimes_k$, and all $k$-algebras have a multiplicative unit. We recall from \cite{g2}, \cite{gs} and \cite{lo} the definition of  the Hochschild cohomology.

 Suppose that $A$ is an associative $k$-algebra (not necessarily commutative), and $M$ is an $A$-bimodule. Define  $C^n(A, M)=Hom_k(A^{\otimes n}, M)$ and $\delta_n:C^n(A,M)\to C^{n+1}(A,M)$ determined by:
\begin{center}$\displaystyle\delta_n(f)(a_1\otimes a_2\otimes ...\otimes a_{n+1})=a_1f(a_2\otimes ...\otimes a_{n+1})+\sum_{i=1}^n(-1)^{i+1}f(a_1\otimes ...\otimes a_{i}a_{i+1}\otimes  ...\otimes a_{n+1})+
(-1)^{n+2}f(a_1\otimes ...\otimes a_{n})a_{n+1}.$
\end{center}
One can show that $\delta_{n+1}\delta_n=0$.
The homology of this complex is denoted by $H^n(A, M)$ and is called the Hochschild cohomology of $A$ with coefficients in $M$.

\subsection{Secondary Cohomology of a Triple $(A,B,\varepsilon)$}
We recall from \cite{sta} the definition of the secondary Hochschild  cohomology.

Let $A$ be an associative $k$-algebra, $B$  a commutative $k$-algebra, $\varepsilon:B\to A$ a morphism of $k$-algebras such that $\varepsilon(B)\subset {\mathcal Z}(A)$,  and $M$ an $A$-bimodule. We assume that for every $\alpha\in B$ and $m\in M$ we have $\varepsilon(\alpha)m=m\varepsilon(\alpha)$. Let  $$C^n((A,B,\varepsilon);M)=Hom_k(A^{\otimes n}\otimes B^{\otimes \frac{n(n-1)}{2}},M).$$
We want to define
$$\delta^{\varepsilon}_n:C^n((A,B,\varepsilon);M)\to C^{n+1}((A,B,\varepsilon);M).$$
It is convenient to think about an element    $T\in A^{\otimes n}\otimes B^{\otimes \frac{n(n-1)}{2}}$ using the following matrix representation:
$$
%\displaystyle\bigotimes
T={\otimes}\left(
\begin{array}{cccccccc}
 a_{1}       & b_{1,2} &...&b_{1,n-2}&b_{1,n-1}&b_{1,n}\\
 1  & a_{2} &...&b_{2,n-2}      &b_{2,n-1}&b_{2,n}\\
 .       &. &...&.&.&.\\
 1& 1  &...&1&a_{n-1}&b_{n-1,n}\\
 1& 1  &...&1&1&a_{n}\\
\end{array}
\right),
$$
where $a_i\in A$, $b_{i,j}\in B$ and $1\in k$. Notice that we do not have exactly the same notation as in \cite{sta}, the difference here is that all the indices are shifted by one.

For $T\in A^{\otimes m+n-1}\otimes B^{\otimes \frac{(m+n-1)(m+n-2)}{2}}$ and for all $0\leq i\leq m-1$ we denote by $T_{i+n}^{i}$ the following "sub-tensor matrix"
\begin{eqnarray*}
T_{i+n}^i=
\displaystyle\otimes
\left(
\begin{array}{ccc}
 a_{i+1}& ... & b_{i+1,i+n}\\
.& ...       & .\\
1&...&a_{i+n}\\
\end{array}
\right).
\end{eqnarray*}
One should notice that unless $i=0$ it does not make sense to talk about $T^i_{i+n}$ as a tensor but only as a sub-tensor of $T$. Clearly we have $T=T^0_n$.

For a tensor matrix $T\in A^{\otimes n}\otimes B^{\otimes \frac{n(n-1)}{2}}$ and positive integers $l,i,$ and $k$ such that $1\leq l\leq i\leq k\leq n-1$ we consider the sub-tensor matrix
 $$M_{i,i+1}^{l,k}=\displaystyle\otimes
\left(
\begin{array}{cccccccc}
 a_l& b_{l,2} &  ...&b_{l,i}b_{l,i+1}& ...& b_{l,k}& b_{l,k+1}\\
1 & a_{l+1}    &...&b_{l+1,i}b_{l+1,i+1}& ...& b_{l+1,k}&b_{l+1,k+1}\\
. & .       &...&...&...&.&.\\
1 & 1      &... &\varepsilon(b_{i,i+1})a_{i}a_{i+1}&...&b_{i,k}b_{i+1,k}&b_{i,k+1}b_{i+1,k+1}\\
. & .       &...&...&...&.&.\\
1 & 1       &...&...&...&a_{k}&b_{k,k+1}\\
1 & 1& ...&...&...&1&a_{k+1}\\
\end{array}
\right).$$

With the above notations we define
$$\delta^{\varepsilon}_n:C^n((A,B,\varepsilon);M)\to C^{n+1}((A,B,\varepsilon);M),$$
\begin{eqnarray*}
&&\delta^{\varepsilon}_n(f)(T_{n+1}^0)=a_1\varepsilon(b_{1,2}b_{1,3}...b_{1,n+1})f(T^1_{n+1})-
f(M^{1,n}_{1,2})+ f(M^{1,n}_{2,3})+\\&&...+
(-1)^{i}f(M^{1,n}_{i,i+1})+...+(-1)^{n-1}f(M^{1,n}_{n-1,n})+(-1)^{n}f(M^{1,n}_{n,n+1})+\\
&&(-1)^{n+1}f(T^0_{n})a_{n+1}\varepsilon(b_{1,n+1}b_{1,n+1}...b_{n,n+1}).
\end{eqnarray*}

\begin{proposition} (\cite{sta})  $(C^n((A,B,\varepsilon);M),\delta_n^{\varepsilon})$ is a complex (i.e. $\delta_{n+1}^{\varepsilon}\delta_n^{\varepsilon}=0$). We denote its homology by $H^n((A,B,\varepsilon);M)$ and we call it the secondary Hochschild cohomology of the triple $(A,B,\varepsilon)$ with coefficients in $M$.
\end{proposition}

\begin{example} When $B=k$ and $\varepsilon:k\to A$ we have that $H^n((A,k,\varepsilon);M)$ is the usual Hochschild cohomology.
\end{example}

\subsection{Pre-Lie systems}
We recall from \cite{g1} the definition of a pre-Lie system.
\begin{definition} A pre-Lie system is a family of pairs $\{V_n,\circ_i\}$, where $V_n$ are $k$-vector spaces for all $n\in \mathbb{Z}$ and $\circ_i=\circ_i(m,n):V_m\otimes V_n\to V_{m+n}$  are $k$-linear maps for all $0\leq i\leq m$. Moreover the following identities hold
\begin{eqnarray*}(f^m\circ_i g^n)\circ_j h^p=(f^m\circ_j h^p)\circ_{i+p}g^n \; \; if \; \; 0\leq j\leq i-1,\\
(f^m\circ_i g^n)\circ_j h^p=f^m\circ_i(g^n\circ_{j-i} h^p) \; \; if \; \; i\leq j\leq n+1.
\end{eqnarray*}
\end{definition}
Given a pre-Lie system $\{V_n,\circ_i\}$ then for all $m\geq 0$  one can define
$\circ:V_m\otimes V_n\to V_{m+n}$
$$f^m\circ g^n=\sum_{i=0}^m(-1)^{ni}f^m\circ_i g^n.$$
The following result was proved in \cite{g1}.
\begin{theorem} Let $\{V_n,\circ_i\}$ be a pre-Lie system. Define $A=\oplus_{n}V_n$ and $[.,.]:A\otimes A\to A$, where $[f^m,g^n]=f^m\circ g^n-(-1)^{mn}[g^n,f^m]$. Then $(A,[.,.])$ is a graded Lie algebra. \label{theorem1}
\end{theorem}

\section{Cup Product and Bracket Product}

\subsection{Cup Product} This section follows closely the results from \cite{g1}. For $f\in C^m((A,B,\varepsilon);A)$ and $g\in C^n((A,B,\varepsilon);A)$ we define
\begin{eqnarray*}
&f\smile g\left(\displaystyle\otimes
\left(
\begin{array}{cccccccc}
 a_{1}& b_{1,2} &  ...&b_{1,m+n-1}&b_{1,m+n}\\
1 & a_{2}    &...&b_{2,m+n-1}&b_{2,m+n}\\
. & .       &...&.&.\\
1& 1& ...&a_{m+n-1}&b_{m+n-1,m+n}\\
1 & 1& ...&1&a_{m+n}\\
\end{array}
\right)\right)=\\
&f\left(\displaystyle\otimes
\left(
\begin{array}{cccccccc}
 a_{1}& b_{1,2} &  ...&b_{1,m-1}&b_{1,m}\\
1 & a_{2}      &...&b_{2,m-1}&b_{2,m}\\
. & . &...&.&.\\
1& 1 &...&a_{m-1}&b_{m-1,m}\\
1 & 1  &...&1&a_{m}\\
\end{array}
\right)\right) \prod\limits_{\substack{m+1\leq j\leq m+n\\1\leq i\leq m}}\varepsilon(b_{i,j})\\
&g\left(\displaystyle\otimes
\left(
\begin{array}{cccccccc}
 a_{m+1}& b_{m+1,m+2} &  ...&b_{m+1,m+n-1}&b_{m+1,m+n}\\
1 & a_{m+2}      &...&b_{m+2,m+n-1}&b_{m+2,m+n}\\
. & . &...&.&.\\
1& 1 &...&a_{m+n-1}&b_{m+n-1,m+n}\\
1 & 1  &...&1&a_{m+n}\\
\end{array}
\right)\right).
\end{eqnarray*}

Using the notations introduced earlier we have the equivalent formula
$$(f\smile g)(T^0_{m+n})=f(T^{0}_m)g(T^m_{m+n})\prod\limits_{\substack{m+1\leq j\leq m+n\\1\leq i\leq m}}\varepsilon(b_{i,j}).$$

One can easily check that $$\smile:C^m((A,B,\varepsilon);A)\otimes C^n((A,B,\varepsilon);A)\to C^{m+n}((A,B,\varepsilon);A)$$
 induces a graded associative algebra structure on $C^*((A,B,\varepsilon);A)$. Moreover, the cup product satisfies the identity
\begin{eqnarray}
\delta^{\varepsilon}_{m+n}(f\smile g)=\delta^{\varepsilon}_m(f)\smile g+(-1)^mf\smile \delta^{\varepsilon}_n(g).
\label{delta1}
\end{eqnarray}

To prove this let $f\in C^m((A,B,\varepsilon);A)$ and $g\in C^n((A,B,\varepsilon);A).$ Then we have

\begin{center}$\delta^{\varepsilon}_{m+n}(f\smile g)(T^0_{m+n+1})=\displaystyle a_1\prod_{i=2}^{m+n+1}\varepsilon(b_{1,i})(f\smile g)(T_{m+n+1}^1)-$\end{center}
\begin{center} $-(f\smile g)(M_{1,2}^{1,m+n})+\dots+(-1)^i(f\smile g)(M_{i, i+1}^{1,m+n})+\dots$\end{center}
\begin{center} $+(-1)^{m+n}(f\smile g)(M_{m+n, m+n+1}^{1,m+n})+$\end{center}

\begin{center} $+(-1)^{m+n+1}(f\smile g)(T_{m+n}^0)a_{m+n+1}\displaystyle\prod_{i=1}^{m+n}\varepsilon(b_{i,m+n+1})= a_1\prod_{i=2}^{m+n+1}\varepsilon(b_{1,i})f(T_{m+1}^1)\prod\limits_{\substack{2\leq i\leq m+1\\m+2\leq j\leq m+n+1}}\varepsilon(b_{i,j})g(T_{m+n+1}^{m+1})-$\end{center}
\begin{center} $-f(M_{1,2}^{1,m})\displaystyle\prod\limits_{\substack{1\leq i\leq m+1\\m+2\leq j\leq m+n+1}}\varepsilon(b_{i,j})g(T_{m+n+1}^{m+1})+\dots+(-1)^if(M_{i,i+1}^{1,m})\prod\limits_{\substack{1\leq i\leq m+1\\m+2\leq j\leq m+n+1}}\varepsilon(b_{i,j})g(T_{m+n+1}^{m+1})+\dots+\displaystyle(-1)^mf(M_{m,m+1}^{1,m})\prod\limits_{\substack{1\leq i\leq m+1\\m+2\leq j\leq m+n+1}}\varepsilon(b_{i,j})g(T_{m+n+1}^{m+1})+(-1)^{m+1}f(T_m^0)g(M_{m+1,m+2}^{m+1,m+n})\prod\limits_{\substack{1\leq i\leq m\\m+1\leq j\leq m+n+1}}\varepsilon(b_{i,j})+(-1)^{m+2}f(T_m^0)g(M_{m+2,m+3}^{m+1,m+n})\prod\limits_{\substack{1\leq i\leq m\\m+1\leq j\leq m+n+1}}\varepsilon(b_{i,j})+\dots
+(-1)^{m+n}f(T_m^0)g(M_{m+n,m+n+1}^{m+1,m+n})\prod\limits_{\substack{1\leq i\leq m\\m+1\leq j\leq m+n+1}}\varepsilon(b_{i,j})+(-1)^{m+n+1}f(T_m^0)g(T_{m+n}^m)a_{m+n+1}\prod\limits_{\substack{1\leq i\leq m\\m+1\leq j\leq m+n}}\varepsilon(b_{i,j})\prod_{i=1}^{m+n}\varepsilon(b_{i,m+n+1}).$ \end{center}

On the other hand we have

\begin{center} $\left(\delta^{\varepsilon}_m(f)\smile g+(-1)^mf\smile \delta^{\varepsilon}_n(g)\right)(T^0_{m+n+1})=$\end{center}

\begin{center}$\displaystyle\delta^{\varepsilon}_m(f)(T_{m+1}^0)g(T_{m+n+1}^{m+1})\prod\limits_{\substack{1\leq i\leq m+1\\m+2\leq j\leq m+n+1}}\varepsilon(b_{i,j})+(-1)^mf(T_m^0)\delta^{\varepsilon}_n(g)(T_{m+n+1}^m)\prod\limits_{\substack{1\leq i\leq m\\m+1\leq j\leq m+n+1}}\varepsilon(b_{i,j})=a_1\prod_{i=2}^{m+1}\varepsilon(b_{1,i})f(T_{m+1}^1)g(T_{m+n+1}^{m+1})\prod\limits_{\substack{1\leq i\leq m+1\\m+2\leq j\leq m+n+1}}\varepsilon(b_{i,j})-\displaystyle-f(M_{1,2}^{1,m})g(T_{m+n+1}^{m+1})\prod\limits_{\substack{1\leq i\leq m+1\\m+2\leq j\leq m+n+1}}\varepsilon(b_{i,j})+\dots+(-1)^if(M_{i,i+1}^{1,m})g(T_{m+n+1}^{m+1})\prod\limits_{\substack{1\leq i\leq m+1\\m+2\leq j\leq m+n+1}}\varepsilon(b_{i,j})+\dots+(-1)^{m}f(M_{m,m+1}^{1,m})g(T_{m+n+1}^{m+1})\prod\limits_{\substack{1\leq i\leq m+1\\m+2\leq j\leq m+n+1}}\varepsilon(b_{i,j})+(-1)^{m+1}f(T_{m}^0)a_{m+1}\displaystyle\prod_{i=1}^m\varepsilon(b_{i,m+1})g(T_{m+n+1}^{m+1})\prod\limits_{\substack{1\leq i\leq m+1\\m+2\leq j\leq m+n+1}}\varepsilon(b_{i,j})+$\end{center}

\begin{center} $\displaystyle(-1)^{m}f(T_m^0)a_{m+1}\prod_{i=m+2}^{m+n+1}\varepsilon(b_{m+1,i})g(T_{m+n+1}^{m+1})\prod\limits_{\substack{1\leq i\leq m\\m+1\leq j\leq m+n+1}}\varepsilon(b_{i,j})+(-1)^{m+1}f(T_m^0)\prod\limits_{\substack{1\leq i\leq m\\m+1\leq j\leq m+n+1}}\varepsilon(b_{i,j})g(M_{m+1,m+2}^{m+1,m+n})+\dots+(-1)^{m}f(T_m^0)\prod\limits_{\substack{1\leq i\leq m\\m+1\leq j\leq m+n+1}}\varepsilon(b_{i,j})
(-1)^{n}g(M_{m+n,m+n+1}^{m+1,m+n})+(-1)^{m}f(T_m^0)(-1)^{n+1}g(T_{m+n}^m)a_{m+n+1}\prod_{i=m+1}^{m+n}\varepsilon(b_{i,m+n+1})
\prod\limits_{\substack{1\leq i\leq m\\m+1\leq j\leq m+n+1}}\varepsilon(b_{i,j}).$\end{center}

One should note that the $(m+1)^{th}$ term in the expansion of $\delta^{\varepsilon}_m(f)\smile g$ and the first term in that of
 $(-1)^mf\smile \delta^{\varepsilon}_n(g)$ cancel each other. It is easy to see now that the terms in the expansion of $\delta^{\varepsilon}_{m+n}(f\smile g)$ are equal to the remaining terms in that of $\delta^{\varepsilon}_m(f)\smile g+(-1)^mf\smile \delta^{\varepsilon}_n(g)$ in the order in which they appear.
Therefore we have the following proposition.
\begin{proposition} The cup product defines a structure of graded associative algebra on the secondary Hochschild cohomology $H^{*}((A,B,\varepsilon);A)$.
$$\smile:H^{*}((A,B,\varepsilon);A)\otimes H^{*}((A,B,\varepsilon);A)\to H^{*}((A,B,\varepsilon);A).$$
\end{proposition}
\begin{proof}
It follows  from (\ref{delta1}).
\end{proof}

\subsection{Bracket Product}

Next we want to define a pre-Lie system. Take $V_m=C^{m+1}((A,B,\varepsilon);A)$ for all $m\geq -1$, and $V_m=0$ for all $m<-1$. Notice that the $m$-cochains have degree $m-1$. Because of this shift it is more convenient to give the definition of $\circ_i:V_{m-1}\otimes V_{n-1}\to V_{m+n-2}$.  For $m$ and $n\geq 0$ and  $0\leq i\leq m-1$ we define
$$\circ_i:C^m((A,B,\varepsilon);A)\otimes C^n((A,B,\varepsilon);A)\to C^{m+n-1}((A,B,\varepsilon);A).$$
If $f^m\in V_{m-1}=C^m((A,B,\varepsilon);A)$, $g^n\in V_{n-1}=C^n((A,B,\varepsilon);A)$ and $0\leq i\leq m-1$ then
\begin{eqnarray*}
&&f^m\circ_i g^n\left(\displaystyle
%\bigotimes
\otimes
\left(
\begin{array}{cccccccc}
 a_{1}& b_{1,2} &  ...&b_{1,m+n-2}&b_{1,m+n-1}\\
1 & a_{2}    &...&b_{2,m+n-2}&b_{2,m+n-1}\\
. & .       &...&.&.\\
1& 1& ...&a_{m+n-2}&b_{m+n-2,m+n-1}\\
1 & 1& ...&1&a_{m+n-1}\\
\end{array}
\right)\right)=
\end{eqnarray*}

\begin{eqnarray*}
&&f^m\left(\displaystyle
%\bigotimes
\otimes
\left(
\begin{array}{ccccccccc}
 a_{1}&  ...&b_{1,i}&\prod\limits_{i< j\leq n+i}b_{1,j}&b_{1,n+i+1}&...&b_{1,m+n-1}\\
1 & ...&b_{2,i}&\prod\limits_{i< j\leq n+i}b_{2,j}&b_{2,n+i+1}&...&b_{2,m+n-1}\\
. & ...&.&.&.&...&.\\
1 & ...&a_{i}&\prod\limits_{i< j\leq n+i}b_{i,j}&b_{i,n+i+1}&...&b_{i,m+n-1}\\
1 & ...&1&g^n\left(T_{n+i}^i\right)&\prod\limits_{i< j\leq n+i}b_{j,n+i}&...&\prod\limits_{i< j\leq n+i}b_{j,m+n-1}\\
1 & ...&1&1&a_{n+i+1}&...&b_{n+i+1,m+n-1}\\
. & ...       &.&.&.&...&.\\
1& ...&1&1&1&...&b_{m+n-2,m+n-1}\\
1 &  ...&1&1&1&...&a_{m+n-1}\\
\end{array}
\right)\right).
\end{eqnarray*}

One can check that for all $0\leq i\leq m-1$ and  $0\leq j\leq i-1$ we have
\begin{eqnarray*}
&&(f^m\circ_i g^n)\circ_j h^p\left(\displaystyle
\otimes
\left(
\begin{array}{cccccccc}
 a_{1}&   ...&b_{1,m+n+p-2}\\
.        &...&.\\
1 & ...&a_{m+n+p-2}\\
\end{array}
\right)\right)=\\
&&\\
&&(f^m\circ_i g^n)\left(\displaystyle
%\bigotimes
\otimes
\left(
\begin{array}{ccccccccc}
 a_{1}&  ...&\prod\limits_{j< k\leq j+p}b_{1,k}&...&b_{1,m+n+p-2}\\
.       &...&.&...&.\\
1        &...&h\left(T_{j+p}^j\right)&...&\prod\limits_{j< k\leq j+p}b_{k,m+n+p-2}\\
.       &...&.&...&.\\
1 & ...&1&...&a_{m+n+p-2}\\
\end{array}
\right)\right)
=\\
&&\\
&&f^m\left(\displaystyle
%\bigotimes
\otimes
\left(
\begin{array}{ccccccccc}
 a_{1}&  ...&\prod\limits_{j< k\leq j+p}b_{1,k}&...&\prod\limits_{i+p-1< l\leq i+p+n-1}b_{1,l}&...&b_{1,m+n+p-2}\\
.       &...&.&...&.&...&.\\
1        &...&h^p\left(T_{j+p}^j\right)&...&\prod\limits_{\substack{j< k\leq n+j\\
i+p-1<l\leq i+p+n-1}}b_{k,l}&...&\prod\limits_{j< k\leq j+p}b_{k,m+n+p-2}\\
.       &...&.&...&.&...&.\\
1        &...&1&...&g^n\left(T_{i+p+n-1}^{i+p-1}\right)&...&\prod\limits_{i+p-1< l\leq i+p+n-1}b_{l,m+n+p-2}\\
.       &...&.&...&.&...&.\\
1 & ...&1&...&1&...&a_{m+n+p-2}\\
\end{array}
\right)\right)
=\\
&&\\
&&(f^m\circ_j h^p)\left(\displaystyle
%\bigotimes
\otimes
\left(
\begin{array}{ccccccccc}
 a_{1}& ...&\prod\limits_{i+p-1< l\leq i+p+n-1}b_{1,l}&...&b_{1,m+n+p-2}\\
.       &...&.&...&.\\
1        &...&g\left(T_{i+p+n-1}^{i+p-1}\right)&...&\prod\limits_{i+p-1< l\leq i+p+n-1}b_{l,m+n+p-2}\\
.       &...&.&...&.\\
1 & ...&1&...&a_{m+n+p-2}\\
\end{array}
\right)\right)
=\\
&&\\
&&(f^m\circ_j h^p)\circ_{i+p-1}g^n \left(\displaystyle
\otimes
\left(
\begin{array}{cccccccc}
 a_{1}&   ...&b_{1,m+n+p-2}\\
.       &...&.\\
1 & ...&a_{m+n+p-2}\\
\end{array}
\right)\right).
%\bigotimes
\end{eqnarray*}
A similar computation shows that for all $i\leq j\leq n$ we have
\begin{eqnarray*}
&&(f^m\circ_i g^n)\circ_j h^p\left(\displaystyle
\otimes
\left(
\begin{array}{cccccccc}
 a_{1} &  ...&b_{1,m+n+p-2}\\
.        &...&.\\
1 & ...&a_{m+n+p-2}\\
\end{array}
\right)\right)=\\
%&&f\left(\displaystyle
%\bigotimes
%\otimes
%\left(
%\begin{array}{ccccccccc}
% a_{1}&  ...&\prod\limits_{i< k\leq i+n+p-1}b_{1,k}&...&b_{1,m+n+p-2}\\
%.       &...&.&...&.\\
%1      &...&g\left(\otimes
%\left(
%\begin{array}{ccccccccc}
% a_{i+1}&  ...&\prod\limits_{j< k\leq j+p}b_{i+1,k}&...&b_{i+1,i+n+p-1}\\
%.       &...&.&...&.\\
%1      &...&h\left(T_{j+p}^j\right)&...&\prod\limits_{j< k\leq j+p}b_{k,i+n+p-1}\\
%.       &...&.&...&.\\
%1 & ...&1&...&a_{i+n+p-1}\\
%\end{array}
%\right)\right)&...&\prod\limits_{i< k\leq i+n+p-1}b_{k,m+n+p-2}\\
%.       &...&.&...&.\\
%1 & ...&1&...&a_{m+n+p-2}\\
%\end{array}
%\right)\right)=\\
&&f^m\circ_i(g^n\circ_{j-i} h^p)\left(\displaystyle
\otimes
\left(
\begin{array}{cccccccc}
 a_{1}&  ...&b_{1,m+n+p-2}\\
.       &...&.\\
1 &  ...&a_{m+n+p-2}\\
\end{array}
\right)\right).
\end{eqnarray*}
Since there is a shift between the degree and the index of the cohomology group, we will write an explicit formula for the pre-Lie algebra structure. Let  $f^m\in H^m((A,B,\varepsilon); A)$ and $g^n\in H^n((A,B,\varepsilon); A).$ We define
$$f^m\circ g^n=\sum_{i=0}^{m-1}(-1)^{(n-1)i}f^m\circ_i g^n.$$ The above considerations imply that we have the following theorem.

\begin{theorem} $\left(H^*((A,B,\varepsilon); A),[.,.]\right)$ is a graded Lie algebra, where
$$[f^m,g^n]=f^m\circ g^n-(-1)^{(m-1)(n-1)}g^n\circ f^m$$
and the degree of $f^m\in H^m((A,B,\varepsilon); A)$ is $m-1$.
\end{theorem}
\begin{proof}
It follows from Theorem \ref{theorem1} and the above computations.
\end{proof}
Define $\pi:A\otimes A\otimes B\to A$ determined by $$\pi(a\otimes b\otimes \alpha)=ab\varepsilon(\alpha).$$
It is easy to see  that  $\delta^{\varepsilon}_1(id_A)=\pi$ and so $\pi$ is a coboundary (of degree $1$). One can also show that
\begin{eqnarray}
f^m\smile g^n=(\pi \circ_0 f^m)\circ_m g^n, \label{equation1}
\end{eqnarray}
and
\begin{eqnarray}
\delta^{\varepsilon}_m(f^m)=[f^m,-\pi]=(-1)^{m-1}[\pi,f^m].\label{equation2}
\end{eqnarray}
At this point the proof of Theorem 3 from \cite{g1} can be used (and we won't reproduce it here) to get the following result:
\begin{theorem} For $f^m\in C^m((A,B,\varepsilon); A)$ and $g^n\in C^n((A,B,\varepsilon); A)$ we have
\begin{eqnarray*}
&f^m\circ \delta^{\varepsilon}(g^n)-\delta^{\varepsilon}(f^m\circ g^n)+(-1)^{n-1}\delta^{\varepsilon}(f^m)\circ g^n=&\\
&(-1)^{n-1}(g^n\smile f^m-(-1)^{mn}f^m\smile g^n).&
\end{eqnarray*}
\end{theorem}
As a simple consequence we obtain:
\begin{corollary} If $f^m\in H^m((A,B,\varepsilon); A)$ and $g^n\in H^n((A,B,\varepsilon); A)$  we have
$$f^m\smile g^n=(-1)^{mn}g^n\smile f^m.$$

\end{corollary}

\section{Extensions of $B$-algebras}

Suppose that $X$ is a $B$-algebra  with $\varepsilon_X:B\to X$ and that there exists a surjective morphism of $B$-algebras $\pi:X\to A$ such that $ker(\pi)^2=0$. Let $M=ker(\pi)$. We require that the $B$-algebra structure induced on $A$ by the map $\pi\circ\varepsilon_X$ coincides with that defined by the map $\varepsilon$. Consider $s:A\to X$ a $k$-linear map such that $\pi s=id_A$.
 Then $M$ is an $A$-bimodule with the multiplication given by
$$am=s(a)m,\; \; \; \; ma=ms(a),$$
for all $m\in M$ and $a\in A$. One can notice that this action does not depend on the choice of the section $s$. Moreover, for all $\alpha \in B$ and all $m\in M$ we have
$$\varepsilon(\alpha)m=m\varepsilon(\alpha)=\varepsilon_X(\alpha)m.$$
As a $k$-vector space we obviously have that $X=s(A)\oplus M$ (that is $s(A)+M=X$ and $s(A)\cap M=0$).

Because of Proposition 2.1 from \cite{sta}, we know that a $B$-algebra structure on $X$ is the same as an associative family of products $m_{\alpha,X}:X\otimes X\to X$ where  $m_{\alpha,X}(x\otimes y)=\varepsilon_X(\alpha)xy$. Since $\pi:X\to A$ is a morphism o $B$-algebras we must have that
 $$\pi(m_{\alpha,X}((s(a)+m)\otimes (s(b)+n)))=m_{\alpha}(\pi(s(a)+m)\otimes \pi(s(b)+n))=\varepsilon(\alpha)ab.$$
Using this and the linearity of the product we get
\begin{eqnarray*}
&&m_{\alpha,X}((s(a)+m)\otimes (s(b)+n))=\\
&&\varepsilon_X(\alpha)(s(a)s(b)+s(a)n+ms(b))=\\
&&s(\varepsilon(\alpha)ab)+\varepsilon(\alpha)an+mb\varepsilon(\alpha) +\varepsilon_X(\alpha)s(a)s(b)-s(\varepsilon(\alpha)ab).
\end{eqnarray*}
One can see that the $k$-linear map $c_s:A\otimes A\otimes B\to M$  defined by
$$c_s(a\otimes b\otimes \alpha)=\varepsilon_X(\alpha)s(a)s(b)-s(\varepsilon(\alpha)ab)$$
is a 2-cocycle. Moreover, if $t:A\to X$ is another section for $\pi$ then
$$\delta^{\varepsilon}_1(s-t)=c_s-c_t.$$
To summarize we have the following result
\begin{lemma} Let $X$ be a $B$-algebra and $\pi:X\to A$ a surjective morphism of $B$-algebras such that $M^2=0$ (where $M=ker(\pi)$). Then $\widehat{c_s}\in H^2((A,B,\varepsilon); M)$ does not depend on the choice of the section $s$. We will denote this element by $c_{X,\pi}$.
\label{lemma4}
\end{lemma}
Next we prove that $c_{X,\pi}$ depends only on the isomorphism class of the extension $0\to M\to X\stackrel{\pi}{\rightarrow} A\to 0$.
\begin{proposition}\label{equiv}
Let $X_1$ and $X_2$ be two $B$-algebras, $\pi_i:X_i\to A$ surjective morphisms of $B$-algebras such that $(ker(\pi_i))^2=0$. Moreover assume that there exists an isomorphism of $B$-algebras $F:X_1\to X_2$ such that $\pi_2\circ F=\pi_1$. Under the identification $M_2=\ker(\pi_2)=F(M_1)$ we have that $c_{X_2,\pi_2}=F^*(c_{X_1,\pi_1})\in H^2((A,B,\varepsilon); M_2)$.
\end{proposition}
\begin{proof}
The proof follows from Lemma \ref{lemma4} and the fact that if $s:A\to X_1$ is a section for $\pi_1$ then $Fs:A\to X_2$ is a section for $\pi_2$.
\end{proof}

In addition, for any $A$-bimodule $M$ such that $\varepsilon(\alpha)m=m\varepsilon(\alpha)$ and for any cocycle $c\in C^2((A, B,\varepsilon); M)$
we can define a $B$-algebra $X$ and a surjective morphism of $B$-algebras  $\pi:X\rightarrow A$ such that $M=ker(\pi)$, $M^2=0$ and $\pi\circ\varepsilon_X=\varepsilon$. To see this we use Proposition 2.1  from \cite{sta} to define a family of products $m_{\alpha,X}:X\otimes X\to X$ as follows. First, we take $X=A\oplus M$, as a $k$-vector space. Second, we define $$m_{\alpha, X}((a+m)\otimes(b+n))=\varepsilon(\alpha)ab+\varepsilon(\alpha)an+mb\varepsilon(\alpha)+c(a\otimes b\otimes \alpha).$$
One can check without any difficulty that $(X, m_{1, X})$ is a $k$-algebra, with unit $1_X=1_A-c(1_A\otimes 1_A\otimes 1_B),$ and that for all $\alpha, \beta\in B$ and $q\in k$ we have $m_{\alpha+\beta, X}=m_{\alpha, X}+m_{\beta, X}$ and $m_{q\alpha, X}=qm_{\alpha, X}.$ The third condition of Proposition 2.1, $m_{\beta\gamma, X}(m_{\alpha, X}\otimes id)=m_{\alpha\beta, X}(id\otimes m_{\gamma, X})$ is equivalent to $c$ being a cocycle and it is satisfied, so $X$ is a $B$-algebra. We have $\varepsilon_X:B\to X$ defined by
$$\varepsilon_X(\alpha)=\varepsilon(\alpha)-2\varepsilon(\alpha)c(1_A \otimes 1_A\otimes 1_B)+c(1_A\otimes 1_A\otimes \alpha).$$

Third, it is clear that the canonical projection $\pi: X\rightarrow A$ is a surjective morphism of $k$-algebras such that $ker(\pi)=M$, $M^2=0$, and that $\pi\circ\varepsilon_X(\alpha)=\varepsilon(\alpha)$. To see that $\pi$ is a morphism of $B$-algebras note that for all $\alpha \in B, a\in A$, and $m\in M$ we have
\begin{eqnarray*}
&&\pi(\alpha(a+m))=\pi(m_{1, X}(\varepsilon_X(\alpha)\otimes(a+m)))=\\
&=&\pi( m_{1, X}(\varepsilon(\alpha)-2\varepsilon(\alpha)c(1_A\otimes 1_A\otimes 1_B)+c(1_A\otimes 1_A\otimes\alpha))\otimes(a+m)))\\ &=&\varepsilon(\alpha) a\\
&=&\alpha\pi(a+m).
\end{eqnarray*}

Finally, we show that the construction of $X$ depends only on the cohomology class of the cocycle $c\in C^2((A, B, \varepsilon); M)$. For this let $c_1, c_2$ be two cocycles in $C^2((A, B, \varepsilon); M)$ such that $c_1-c_2=\delta_1^\varepsilon f$, where $f:A\rightarrow M$ is $k$-linear. Denote by $X_1$ and $X_2$ the $B$-algebras defined by the cocycles $c_1$ and $c_2$, by $m_{\alpha, X_1}$ and $m_{\alpha, X_2}$ their corresponding families of products, and by $\pi_1$ and $\pi_2$ the canonical projections of $X_1$ and $X_2$ onto $A$. Note that by construction $X_1=X_2=A\oplus M$ as $k$-vector spaces. Then the map $F: X_1\rightarrow X_2$, defined by the formula $F(a+m)=a+m+f(a)$ is an isomorphism of $B$-algebras such that $\pi_2\circ F=\pi_1$. It is easy to see that $F$ is an isomorphism of $k$-algebras such that $\pi_2\circ F=\pi_1$, so we will only prove that $F$ is $B$-linear. Indeed, for $\alpha\in B, a\in A$ and $m\in M$ we have
\begin{eqnarray*}
&&F(\alpha (a+m))=F(m_{1, X_1}(\varepsilon_{X_1}(\alpha)\otimes (a+m)))\\
&&=F(m_{1, X_1}(\varepsilon(\alpha)-2\varepsilon(\alpha)c_1(1\otimes 1\otimes 1)+c_1(1\otimes 1\otimes\alpha))\otimes(a+m)))\\
&&=F(\varepsilon(\alpha)a+\varepsilon(\alpha)m-2\varepsilon(\alpha)c_1(1\otimes 1\otimes 1)a+c_1(1\otimes 1\otimes\alpha)a+c_1(\varepsilon(\alpha)\otimes a\otimes 1))\\
&&=\varepsilon(\alpha)a+\varepsilon(\alpha)m-2\varepsilon(\alpha)c_1(1\otimes 1\otimes 1)a+c_1(1\otimes 1\otimes\alpha)a+c_1(\varepsilon(\alpha)\otimes a\otimes 1)\\
&&+f(\varepsilon(\alpha)a).
\end{eqnarray*}

On the other hand we have
\begin{eqnarray*}
&&\alpha F(a+m)=m_{1,X_2}(\varepsilon_{X_2}(\alpha)\otimes(a+m+f(a)))\\
&&=m_{1, X_2}((\varepsilon(\alpha)-2\varepsilon(\alpha)c_2(1\otimes 1\otimes 1)+c_2(1\otimes 1\otimes\alpha))\otimes (a+m+f(a)))\\
&&=\varepsilon(\alpha)a+\varepsilon(\alpha)m+\varepsilon(\alpha)f(a)-2\varepsilon(\alpha)c_2(1\otimes 1\otimes 1)a+c_2(1\otimes 1\otimes\alpha)a\\&&+c_2(\varepsilon(\alpha)\otimes a\otimes 1).
\end{eqnarray*}

Thus we get $$F(\alpha (a+m))-\alpha F(a+m)=2\varepsilon(\alpha)(c_2(1\otimes 1\otimes 1)-c_1(1\otimes 1\otimes 1))a+$$
$$+(c_1(1\otimes 1\otimes\alpha)-c_2(1\otimes 1\otimes\alpha))a+(c_1(\varepsilon(\alpha)\otimes a\otimes 1)-c_2(\varepsilon(\alpha)\otimes a\otimes 1))-\varepsilon(\alpha)f(a)+f(\varepsilon(\alpha)a). $$

Since $c_1-c_2=\delta_1^\varepsilon f$ we have the following identities $$c_2(1\otimes 1\otimes 1)-c_1(1\otimes 1\otimes 1)=-f(1)$$
$$c_1(1\otimes 1\otimes\alpha)-c_2(1\otimes 1\otimes\alpha)=2\varepsilon(\alpha)f(1)-f(\varepsilon(\alpha))$$
$$c_1(\varepsilon(\alpha)\otimes a\otimes 1)-c_2(\varepsilon(\alpha)\otimes a\otimes 1)=\varepsilon(\alpha)f(a)-f(\varepsilon(\alpha)a)+f(\varepsilon(\alpha))a.$$

Therefore we obtain that $F(\alpha (a+m))-\alpha F(a+m)=0$, so $F$ is an isomorphism of $B$-algebras such that $\pi_2\circ F=\pi_1$.

Assume now that we have an extension given by the following data: a morphism of $k$-algebras $\varepsilon_X:B\to X$; a surjective morphism of $B$-algebras $\pi:X\to A$ such that $ker(\pi)^2=0$, $M=ker(\pi)$; $\pi\circ\varepsilon_X=\varepsilon$; and a $k$-linear map $s:A\to X$ such that $\pi s=id_A$. If we consider the cocycle $c_s\in C^2((A, B,\varepsilon); M)$ defined earlier and then we consider the extension associated to this cocycle then it is not hard to see that we obtain an extension equivalent to the initial one. Similarly,  given an $A$-bimodule $M$ such that $\varepsilon(a)m=m\varepsilon(a)$ and  a cocycle $c\in C^2((A, B, \varepsilon); M)$ we construct the extension associated to $c$. If we now take the cocycle $c_s$ determined by a section $s:A\to X$ with $\pi s=id_A$ then we have that $c_s-c=\delta^\varepsilon_1 u$, where $u:A\to M$ is the $k$-linear map induced by $s$ on $M$. Indeed, we have that $c_s(a\otimes b\otimes\alpha)= c(\varepsilon(\alpha)\otimes ab\otimes 1)+c(1\otimes 1\otimes\alpha)ab-2\varepsilon(\alpha)c(1\otimes 1\otimes 1)ab+\varepsilon(\alpha)c(a\otimes b\otimes 1)+ \delta u(a\otimes b\otimes\alpha)$ for all $a, b\in A$ and $\alpha\in B$. The key observation here is that the cocycle condition implies that $c(a\otimes b\otimes\alpha)=c(\varepsilon(\alpha)\otimes ab\otimes 1)+c(1\otimes 1\otimes\alpha)ab-2\varepsilon(\alpha)c(1\otimes 1\otimes 1)ab+\varepsilon(\alpha)c(a\otimes b\otimes 1).$

The above considerations allow us to conclude that $H^2((A, B, \varepsilon); M)$ can be naturally identified with the equivalence classes of extensions of $B$-algebras of $A$ by $M$, for any $A$-bimodule $M$ such that $\varepsilon(\alpha)m=m\varepsilon(\alpha)$.

\section{A Hodge Type Decomposition of the Secondary Cohomology}
In this section we will assume that $A$ is commutative, $k$ is a field of characteristic 0, and $M$ is a symmetric $A$-bimodule ($i.e.$ $ am=ma$ for all $a\in A$ and $m\in M$). We denote by $kS_n$ the group algebra of the group of permutations of $n$ objects. Under these conditions Barr proved in \cite{B}  that $kS_n$ operates on the
$n$-cochains, $C^n(A, M)$, of the complex defining the Hochschild cohomology of $A$ with coefficients in $M$ and that there is a non-central idempotent $e_n\in\mathbb{Q}S_n$ such that $\delta_n(e_nf)=e_{n+1}(\delta_n f)$. This implies that the Hochschild
complex is a direct sum of two sub-complexes, corresponding to $e_n$ and $1-e_n$. Barr's ideas were extended in \cite{gs2} by  Gerstenhaber and  Schack
who showed that $\mathbb{Q}S_n$ contains $n$ mutually orthogonal idempotents $e_n(1), e_n(2), \dots, e_n(n)$ which sum to the identity and with the property that for each cochain $f\in C^n(A, M)$ we have $\delta_n(e_n(k)f)=e_{n+1}(k)(\delta _nf).$  From this it follows that the Hochschild cohomology $H^n(A, M)$ has a Hodge type decomposition into a direct sum of $n$ summands.  Barr's original idempotent $e_n$ is $e_n(1)$ and the idempotents and the decomposition are labeled BGS (Barr-Gerstenhaber-Schack). The action of $S_n$ on the $n$-cochains $C^n(A, M)$ is given by $$(\pi f)(a_1\otimes a_2\otimes\dots\otimes a_n)=(f\pi^{-1})(a_1\otimes a_2\otimes\dots\otimes a_n)=f(a_{\pi(1)}\otimes a_{\pi(2)}\otimes\dots\otimes a_{\pi(n)}).$$

It is not hard to see that $S_n$ acts on the $n$-cochains of the secondary cohomology. Indeed, for $\pi\in S_n$ and $f\in C^n((A, B, \varepsilon); M)$ we define the left action of $S_n$ by setting
\begin{center} $(\pi f)\left(\displaystyle\otimes
\left(
\begin{array}{cccccccc}
 a_{1}& b_{1,2} &  ...&b_{1,n}\\
1 & a_{2}       & ...&b_{2,n}\\
%1& 1  & ...     &b_{3,n}\
. & .       &...&.\\
1 & 1& ...&a_{n}\\
\end{array}
\right)\right)=f\left(\displaystyle\otimes
\left(
\begin{array}{cccccccc}
 a_{\pi(1)}& b_{\pi(1, 2)} &  ...&b_{\pi(1, n)}\\
1 & a_{\pi(2)}       & ...&b_{\pi(2, n)}\\
%1& 1  & ...     &b_{3,n}\\
. & .       &...&.\\
1& 1& ...&a_{\pi(n)}\\
\end{array}
\right)\right),$ \end{center}
where, for each $1\leq i < j\leq n,$ the element $b_{\pi(i, j)}$ is equal to $b_{\pi(i), \pi(j)}$ if $\pi(i)<\pi(j)$ and equal to $b_{\pi(j),\pi(i)}$ if $\pi(j)<\pi(i)$. Similarly, one defines the right action of $S_n$ on $C^n((A, B, \varepsilon); M)$ by using $\pi^{-1}$. It is important to note that the order
of the elements $a_{\pi(1)}, a_{\pi(2)}, \dots, a_{\pi(n)}$ on the diagonal of the above tensor matrix determines completely the positions of $b_{\pi(i, j)}$.

We want to show that for $f\in C^n((A, B, \varepsilon); M)$ we have that $\delta^{\varepsilon}_n(e_n(k)f)=e_{n+1}(k)(\delta^{\varepsilon}_n f)$. This will imply that the secondary cohomology $H^\bullet((A, B, \varepsilon); M)$ has a Hodge type decomposition. For this we use that the BGS idempotents $e_n(1), e_n(1),\dots ,e_n(n)$ are polynomials, with rational coefficients, of the total shuffle operator.

Following Barr \cite{B}, for $0<r<n$ and $\pi\in S_n$ we say that $\pi$ is a pure shuffle of $r$ through $n-r$ if $\pi(1)<\dots<\pi(r)$ and $\pi(r+1)<\dots <\pi(n)$. Then the $r^{th}$ shuffle operator is $s_{r, n-r}=\sum\limits_{\substack{\mathrm{pure} \\ \mathrm{shuffles}}}(-1)^\pi\pi$, where $(-1)^\pi$ is the sign of $\pi$. The total shuffle operator is defined by $s_n=\sum\limits_{1\leq r\leq n-1}s_{r, n-r}$ and satisfies $\delta_n(s_nf)=s_{n+1}(\delta_n f)$, for all $f\in C^n(A, M)$. Moreover, Gerstenhaber and Schack showed in \cite{gs2} that the minimal polynomial of $s_n$ over $\mathbb{Q}$ is $\mu_n(x)=\prod\limits_{1\leq i\leq n}[x-(2^i-2)]$. They defined \begin{center} $e_n(k)=\prod\limits_{\substack{1\leq i\leq n\\ i\neq k}}(\lambda_k-\lambda_i)^{-1}\prod\limits_{\substack{1\leq i\leq n\\i\neq k}}(s_n-\lambda_i), $\;$ \mathrm{where} $\;$ \lambda_i=2^i-2.$\end{center}

 We want to justify that for every $f\in C^n((A, B, \varepsilon), M)$ we have $$\left(\delta_n^\varepsilon(s_n f)-s_{n+1}(\delta_n^\varepsilon f)\right)\left(\otimes\left(\begin{array}{cccccccc}
 a_{1}& b_{1,2} &  ...&b_{1,n+1}\\
1 & a_{2}       & ...&b_{2,n+1}\\
%1& 1  & ...     &b_{3,n}\
. & .       &...&.\\
1 & 1& ...&a_{n+1}\\
\end{array}
\right)\right)=0,$$ for $a_1, a_2, \dots, a_{n+1}\in A$ and $b_{ij}\in B$, $1\leq i<j\leq n+1$. The expansion of the left side shows that the
identity holds for $b_{i,j}=1$, a direct consequence of $\delta_n(s_n \bar{f})-s_{n+1}(\delta_n \bar{f})=0$ (where $\bar{f}$ is obtained from $f$ by taking $b_{i,j}=1$). This means that the diagonals of the tensor sub-matrices of types $T_1^{n+1}, T_n^0,$ and $M_{i, i+1}^{1,n}$ in the expansion of $\delta_n^\varepsilon(s_n f)-s_{n+1}(\delta_n^\varepsilon f)$ appear in identical pairs and with opposite signs. But, as a consequence of way we defined the action of $S_n$ on the secondary cochains and of the definition of $\delta_n^\varepsilon$,  the order
of the elements $a_{\pi(1)}, a_{\pi(2)}, \dots, a_{\pi(n+1)}$ and of the products $a_{\pi(i)}a_{\pi(j)}\varepsilon(b_{\pi(i,j)})$ on the diagonal of the above tensor matrices determines completely the positions of all $b_{\pi(i, j)}$ and their products in $T_1^{n+1}, T_n^0,$ and $M_{i, i+1}^{1,n}$.
This implies that $\delta_n^\varepsilon(s_n f)=s_{n+1}(\delta_n^\varepsilon f)$.

In addition, because $\mu_n(s_n)=0$, we have the identity \begin{center} $\delta_n^\varepsilon(\mu_n (s_n) f)=\prod\limits_{\substack{1\leq i\leq n}}(s_{n+1}-\lambda_i)(\delta_n^\varepsilon f)=0,$\end{center} so we get that
\begin{center} $\delta_n^\varepsilon(e_n(k)f)=\prod\limits_{\substack{1\leq i\leq n\\i\neq k}}(\lambda_k-\lambda_i)^{-1}\prod\limits_{\substack{1\leq i\leq n\\i\neq k}}(s_{n+1}-\lambda_i)(\delta_n^\varepsilon f)=$\end{center}

\begin{center} $=\prod\limits_{\substack{1\leq i\leq n+1\\ i\neq k}}(\lambda_k-\lambda_i)^{-1}\prod\limits_{\substack{1\leq i\leq n\\ i\neq k}}(s_{n+1}-\lambda_i)
(\lambda_k-\lambda_{n+1}+s_{n+1}-\lambda_k)(\delta_n^\varepsilon f)=e_{n+1}(k)(\delta_n^\varepsilon f).$ \end{center}

Adopting the notations from \cite{gs2}, each idempotent $e_n(k)$ determines a submodule
of $C^n((A, B, \varepsilon); M)$, namely $$C^{k, n-k}((A, B, \varepsilon); M)=e_n(k)C^n((A, B, \varepsilon); M).$$
By setting $e_n(k)=0$ if $k>n$, $e_n(0)=0$ if $n\neq 0$, and $e_0(0)=1$ we have that the complex defining the secondary cohomology
decomposes as \begin{center} $C^\bullet ((A, B, \varepsilon); M)=\coprod\limits_{k\geq 0}C^{k, \bullet-k}((A, B, \varepsilon); M)=\coprod\limits_{k\geq 0}e_n(k)C^\bullet((A, B, \varepsilon); M).$ \end{center}
Denoting by $H^{k, \bullet-k}((A, B, \varepsilon); M)$ the homology of the complex $C^{k, \bullet-k}((A, B, \varepsilon); M)$  we have the following

\begin{theorem} If $\varepsilon: B\rightarrow A$ is a morphism of commutative $k$-algebras, $\mathbb{Q}\subset k$, and $M$ is a symmetric
$A$-bimodule then \begin{center} $H^\bullet((A, B, \varepsilon); M)=\coprod\limits_{k\geq 0}H^{k, \bullet-k}((A, B, \varepsilon); M)$. \end{center}

\end{theorem}

\section{Some Examples}
It was noticed in \cite{sta} that  there  exists a natural morphism $$\Phi_n:H^n((A,B,\varepsilon);M)\to H^n(A,M),$$
induced by the inclusion $i:A^{\otimes n}\to A^{\otimes n}\otimes B^{\otimes \frac{n(n-1)}{2}}$,
$$i_n(a_1\otimes ...\otimes a_n)=\displaystyle\otimes
\left(
\begin{array}{cccccccc}
 a_{1}& 1 &  ...&1&1\\
1 & a_{2}    &...&1&1\\
. & .       &...&.&.\\
1& 1& ...&a_{n-1}&1\\
1 & 1& ...&1&a_{n}\\
\end{array}
\right)
$$
In this section we will see that in general $\Phi_n$ is neither onto nor one to one.

First, notice that if $u:A\to M$ is $k$-linear such that $\delta^{\varepsilon}_1(u)=0$ then we must have that $a\varepsilon(\alpha)u(b)-u(ab\varepsilon(\alpha))+u(a)b\varepsilon(\alpha)=0$. This implies that $\Phi_1(u)$ is a derivation that is $B$-linear. Since, in general, not all $k$-derivations of $A$ are $B$-linear we get that $\Phi_1$ is not necessarily onto. We have the following result
\begin{proposition}
$$H^0((A,B,\varepsilon); M)= M^A,$$  $$H^1((A,B,\varepsilon); M)= Der_B(A,M)/Inn(A,M).$$
\end{proposition}
\begin{proof} Straightforward computation.
\end{proof}

\begin{proposition} Let $\Phi_2: H^2((A,B,\varepsilon); M)\to H^2(A,M)$. If on $M$ we consider the $B$-bimodule structure induced by $\varepsilon$, then there exists an isomorphism
$$\chi: \frac{Der_k(B,M)}{\varepsilon^*(Der_k(A,M))}\to ker(\Phi_2)$$
determined by $\chi(u)(a\otimes b\otimes \alpha)=au(\alpha)b$.
\label{prop6}
\end{proposition}
\begin{proof}
Let $\sigma\in Z^2((A,B,\varepsilon); M)$ such that $\Phi_2(\widehat{\sigma})=0\in H^2(A,M)$. This means that there exists a $k$-linear map $u:A\to M$ such that $$\sigma(a\otimes b\otimes 1)=\delta_1(u)(a\otimes b)=au(b)-u(ab)+au(b).$$
We consider the element $\tau\in Z^2((A,B,\varepsilon); M)$, $\tau=\sigma-\delta_1^{\varepsilon}(u).$ Obviously we have that $\widehat{\sigma}=\widehat{\tau}\in H^2((A,B,\varepsilon); M)$, and $\tau\left(\displaystyle\otimes
\left(
\begin{array}{ccc}
a& 1\\
1 & b\\
\end{array}
\right)\right)=0.$

Since $\tau\in Z^2((A,B,\varepsilon); M)$, we have
\begin{eqnarray*}
&a\varepsilon(\alpha\beta)\tau\left(\displaystyle\otimes
\left(
\begin{array}{ccc}
b& \gamma\\
1 & c\\
\end{array}
\right)\right)-\tau\left(\displaystyle\otimes
\left(
\begin{array}{ccc}
ab\varepsilon(\alpha)& \beta\gamma\\
1 & c\\
\end{array}
\right)\right)
+\tau\left(\displaystyle\otimes
\left(
\begin{array}{ccc}
a& \alpha\beta\\
1 & bc\varepsilon(\gamma)\\
\end{array}
\right)\right)&\\
&-\tau\left(\displaystyle\otimes
\left(
\begin{array}{ccc}
a& \alpha\\
1 & b\\
\end{array}
\right)\right)c\varepsilon(\beta\gamma)=0.&
\end{eqnarray*}
When $\alpha=\beta=1$ we have:
$$a\tau\left(\displaystyle\otimes
\left(
\begin{array}{ccc}
b& \gamma\\
1 & c\\
\end{array}
\right)\right)=\tau\left(\displaystyle\otimes
\left(
\begin{array}{ccc}
ab& \gamma\\
1 & c\\
\end{array}
\right)\right),$$
and similarly when $\beta=\gamma=1$
$$\tau\left(\displaystyle\otimes
\left(
\begin{array}{ccc}
a& \alpha\\
1 & bc\\
\end{array}
\right)\right)=\tau\left(\displaystyle\otimes
\left(
\begin{array}{ccc}
a& \alpha\\
1 & b\\
\end{array}
\right)\right)c.$$
If we define $v:B\to M$ by $v(\alpha)=\tau\left(\displaystyle\otimes
\left(
\begin{array}{ccc}
1& \alpha\\
1 & 1\\
\end{array}
\right)\right)$ then we get:
$$\tau\left(\displaystyle\otimes
\left(
\begin{array}{ccc}
a& \alpha\\
1 & b\\
\end{array}
\right)\right)=a\tau\left(\displaystyle\otimes
\left(
\begin{array}{ccc}
1& \alpha\\
1 & 1\\
\end{array}
\right)\right)b=av(\alpha)b.$$
We will denote the $2$-cocycle $\tau$ by $\sigma_{v}$.
One can easily check that $v(\alpha\beta)=\varepsilon(\alpha)v(\beta)+v(\alpha)\varepsilon(\beta)$ (i.e. $v\in Der_k(B,M)$).

If $\sigma_v=\delta_1^{\varepsilon}(w)$ for some $w:A\to M$, then we must have
\begin{eqnarray}
av(\alpha)b=a\varepsilon(\alpha)w(b)-w(a\varepsilon(\alpha)b)+w(a)\varepsilon(\alpha)b. \label{w1}
\end{eqnarray}
For $\alpha=1$ we get that $w(ab)=aw(b)+w(a)b$ and so $w\in Der_k(A,M)$. If  in  equation (\ref{w1}) we take $a=b=1$ then we have $$v(\alpha)=w(\varepsilon(\alpha)),$$
which concludes our proof.
\end{proof}

Next, we want to show that $\Phi_2$ need not be one to one. For this let $A=M=k[X],$ $f(X)\in k[X]$, $B=k[f]$, and let $\varepsilon:B\to A$, $\varepsilon(f)=f(X)$.

For $q(X)\in k[X]$ we consider $\sigma_{q(X)}:A\otimes A\otimes B\to A$ defined by $$\sigma_{q(X)}(P(X)\otimes Q(X)\otimes \alpha(f(X)))=q(X)P(X)Q(X)\alpha'(f(X)).$$
One can see that $\delta^{\varepsilon}_2(\sigma_{q(X)})=0.$  Since $H^2(A, A)=0$ we have that $H^2((A, B, \varepsilon); M)=ker(\Phi_2)$, so every $\hat{\sigma}\in H^2((A, B, \varepsilon); M)$ is of the form $\hat{\sigma}(a\otimes b\otimes\alpha)=av(\alpha)b$, for $v\in Der_k(B, M)$. With this remark we can prove the following result:

\begin{proposition}
Let $\widehat{\sigma} \in
H^2((A, B,\varepsilon);M)$ then there exists $q(X)\in k[X]$ such that $\widehat{\sigma}=\widehat{\sigma_{q(X)}}$. Moreover, if $p(X)$, $q(X)\in k[X]$ then  $\widehat{\sigma_{q(X)}}= \widehat{\sigma_{p(X)}}\in H^2((A, B,\varepsilon);M)$  if and only if $\widehat{p(X)}=\widehat{q(X)}\in k[X]/<f'(X)>$.
\label{prop4}
\end{proposition}
\begin{proof}  On $M=k[X]$ we have the $k[f]$-bimodule structure determined by $f\cdot P(X)=f(X)P(X)$. Let $u\in Der_k(B,M)$ and take $q(X)=u(f)$. Then $u(\Lambda(f))=\Lambda'(f(X))q(X)$.

Let $t\in Der_k(A,M)$, and take $t(X)=r(X)\in k[X]$. We have that $t(P(X))=P'(X)r(X)$ and so $t(\varepsilon(\Lambda(f)))=t(\Lambda(f(X)))=\Lambda'(f(X))f'(X)r(X)$. Now the result follows directly from Proposition \ref{prop6}.
\end{proof}

\begin{remark}If $f(X)\in k[X]$ has the property that the ideal generated by $f'(X)$ is not trivial then the map $\Phi$ is not one to one. Take for example $n\geq 2$ and $f(X)=X^n$ such that $n$ does not divide the characteristic of $k$. Then  we have that   $dim_k(H^2((k[X],k[X^n],\varepsilon);k[X]))= n-1$. %It seems reasonable to believe that this is actually an equality, but at this point we don't have a proof for that.
\end{remark}
\begin{remark} Using the results from \cite{sta}, one can notice that the element  $\widehat{\sigma_{p(X)}}\in H^2((A,B,\varepsilon);M)$ corresponds to the $B$-algebra structure on $A[[t]]$ defined by the  morphism $\varepsilon_t: k[f(X)]\to k[X][[t]]$ where  $\varepsilon_t(f(X))=f(X)+tp(X)$.
\end{remark}

More generally, consider $A=M=k[X,Y]$.
Let  $f(X,Y)$ and $g(X,Y)\in A=k[X,Y]$, take  $B=k[f,g]$ and define $\varepsilon:k[f,g]\to k[X,Y]$ determined by $\varepsilon(f)=f(X,Y)$ and  $\varepsilon(g)=g(X,Y)$. For any $a(X,Y)$ and $b(X,Y)\in k[X,Y]$ we can define $\sigma_{a,b}:A\otimes A\otimes B\to A$ by
\begin{eqnarray*}
&\sigma_{a,b}(P(X,Y)\otimes Q(X,Y)\otimes \Lambda(f,g))=&\\
&P(X,Y)Q(X,Y)(\frac{\partial \Lambda}{\partial f}(f(X,Y),g(X,Y))a(X,Y)+\frac{\partial\Lambda}{\partial g}(f(X,Y),g(X,Y))b(X,Y))&
\end{eqnarray*}
for all $P(X,Y)$, $Q(X,Y)\in k[X,Y]$ and $\Lambda(f,g)\in k[f,g]$.

\begin{proposition}
Let $\widehat{\sigma} \in ker(\Phi_2:
H^2((A, B,\varepsilon);M)\to H^2(A,M))$ then there exist $a(X,Y)$, $b(X,Y)\in k[X,Y]$ such that $\widehat{\sigma}=\widehat{\sigma_{a,b}}$. Moreover,  $\widehat{\sigma_{a,b}}=\widehat{\sigma_{c,d}}\in H^2((A,B,\varepsilon),A)$ if and only if there exist $v(X,Y)$ and $w(X,Y)\in k[X,Y]$ such that
\begin{eqnarray*}
\left(
\begin{array}{cccccccc}
a(X,Y)-c(X,Y)\\
b(X,Y)-d(X,Y)
\end{array}
\right)=
\left(
\begin{array}{cccccccc}
\frac{\partial f}{\partial X}(X,Y)& \frac{\partial f}{\partial Y}(X,Y)\\
\frac{\partial g}{\partial X}(X,Y)&\frac{\partial g}{\partial Y}(X,Y)
\end{array}
\right)
\left(
\begin{array}{cccccccc}
v(X,Y)\\
w(X,Y)
\end{array}
\right)
\end{eqnarray*}
\label{prop5}
\end{proposition}
\begin{proof} The proof is similar with that of Proposition \ref{prop4}. On $M=k[X,Y]$ we have the $k[f,g]$-bimodule structure determined by $f\cdot P(X,Y)=f(X,Y)P(X,Y)$ and $g\cdot P(X,Y)=g(X,Y)P(X,Y)$. Let $u\in Der_k(B,M)$ and take $a(X,Y)=u(f)$ and $b(X,Y)=u(g)$, then
\begin{eqnarray*}
&u(\Lambda(f,g))=\frac{\partial \Lambda}{\partial f}(f(X,Y),g(X,Y))a(X,Y)+\frac{\partial\Lambda}{\partial g}(f(X,Y),g(X,Y))b(X,Y)).&
\end{eqnarray*}
Let $t\in Der_k(A,M)$, and take $t(X)=v(X,Y)$ and $t(Y)=w(X,Y)\in k[X,Y]$. We have that $t(P(X,Y))= \frac{\partial P}{\partial X}(X,Y)u(X,Y)+\frac{\partial P}{\partial Y}(X,Y)v(X,Y)$ and so
\begin{eqnarray*}
&t(\varepsilon(f))=t(f(X,Y))= \frac{\partial f}{\partial X}(X,Y)v(X,Y)+\frac{\partial f}{\partial Y}(X,Y)w(X,Y),&\\
&t(\varepsilon(g))=t(g(X,Y))= \frac{\partial g}{\partial X}(X,Y)v(X,Y)+\frac{\partial g}{\partial Y}(X,Y)w(X,Y).&
\end{eqnarray*}
 Now the result follows directly from Proposition \ref{prop6}.
\end{proof}

\begin{remark}
\label{Jac} A similar statement can be proved if we take  $A=k[X_1,...,X_n]$, $B=k[f_1,...,f_n]$ and $\varepsilon(f_i)=f_i(X_1,...X_n)\in k[X_1,...,X_n]$.
\end{remark}

\begin{remark} In Proposition \ref{prop5} we  proved that the subspace $ker(\Phi_2)$ of  $H^2((A,B,\varepsilon); A)$ is isomorphic with $(k[X,Y]\oplus k[X,Y])/Image(J(f,g))$, where $J(f,g):k[X,Y]\oplus k[X,Y]\to k[X,Y]\oplus k[X,Y]$ is determined by the Jacobian matrix associated to the pair $(f(X,Y),g(X,Y))$.

When $k$ is a field with $char(k)=p$, $f(X,Y)=X+X^p$ and $g(X,Y)=Y+Y^p$ then one can see that $Image(J(f,g))=k[X,Y]\oplus k[X,Y]$ and $\varepsilon$ is not onto. It is possible to have $ker(\Phi)=0$ without the map $\varepsilon$ being surjective. However, when $char(k)=0$ we can give the following reformulation, for polynomials in two variables, of the Jacobian problem stated in \cite{W} ($n$ variables if we consider \ref{Jac}).
\end{remark}
\begin{conjecture} Let $k$ be a field, $char(k)=0$. Take  $A=k[X,Y]$, $B=k[f,g]$, $\varepsilon(f)=f(X,Y)$ and $\varepsilon(g)=g(X,Y)$. If $\Phi_2 :H^2((A,B,\varepsilon);A)\to H^2(A,A)$ is one to one, then $\varepsilon$ is surjective.
\end{conjecture}

\begin{remark} Notice that from Proposition \ref{prop6} we have an exact sequence:
$$H^1(A,M)\stackrel{\varepsilon^*}{\rightarrow}H^1(B,M)\stackrel{\chi}{\rightarrow} H^2((A,B,\varepsilon);M)\stackrel{\Phi_2}{\rightarrow}H^2(A,M).$$
It is reasonable to belive that this can be extended to a long exact sequence. Also, one can ask if the secondary cohomology can be seen as a derived functor (Ext functor) in an appropriate category. We are planing to investigate these problems in a follow up paper.
\end{remark}

%%%%%%%%%%%%%%%
%\section*{Acknowledgment}

%%%%%%%%%%%%%%%%%%%%%%%%%%%%%
%%%%%%%%%%%%%%%%%%%%%%%%%%%%%%%%%%%%%%%
%%%%%%%%

%%%%%%%%%%

\bibliographystyle{amsalpha}

\begin{thebibliography}{A}
\bibitem
[B]
{B}
M. Barr, \textit{Harrison Homology, Hochschild Homology and Triples}, Journal of Algebra {\bf 8}
(1968), 314--323.

\bibitem
[G1]
{g1}
M. Gerstenhaber, \textit{The Cohomology Structure of an Associative Ring}, Ann. of Math. (2) {\bf 78} (1963), 267--288.

\bibitem
[G2]
{g2}
M. Gerstenhaber, \textit{On the Deformation of Rings and Algebras}, Ann. of Math. (2)
{\bf 79} (1964), 57--103.


\bibitem
[GS1]
{gs}
M. Gerstenhaber and S. D. Schack, \textit{Algebraic Cohomology and Deformation Theory}, Kluwer Acad. Publ., Dordrecht,  NATO Adv. Sci. Inst. Ser. C Math. Phys. Sci.,   {\bf 247} (1988), 11--264.

\bibitem
[GS2]
{gs2}
M. Gerstenhaber and S. D. Schack, \textit{A Hodge-type Decomposition for Commutative Algebra Cohomology}, Journal of Pure and Applied Algebra   {\bf 48} (1987), 229--247.

\bibitem
[H]
{h} G. Hochschild, \textit{On the Cohomology Groups of an Associative Algebra}, Ann. of Math. (2) {\bf 46} (1945), 58–-67.

\bibitem
[L]
{lo} J. L. Loday, \textit{Cyclic Homology}, Springer-Verlag, Grundlehren der mathematischen Wissenschaften, {\bf 301} (1992).

\bibitem
[M]
{may} J. P. May, \textit{Simplicial Objects in Algebraic Topology}, Chicago Lectures in Mathematics, (1967).


\bibitem
[S1]
{staic} M. D. Staic, \textit{Secondary Cohomology and k-invariants}, B. Belg. Math. Soc., {\bf 19} (2012),  561--572.



%\bibitem
%[S2]
%{sm} M. D. Staic,
%\textit{An explicit description of the simplicial group $K(A,n)$.}  J. Austr. Math Soc., {\bf 95} no. 1 (2013),  133--144.

\bibitem
[S2]
{sta} M. D. Staic, \textit{Secondary Hochschild Cohomology}, preprint (2013), arXiv:1311.7124.

\bibitem
[W]
{W}
S. S-S Wang, \textit{A Jacobian Crirerion for Separability}, J. of Algebra,
{\bf 65} (1980), 453-494.


\end{thebibliography}

\end{document}